\documentclass[reqno]{amsart}

\usepackage[latin1]{inputenc}
\usepackage{amssymb}
\usepackage{graphicx}
\usepackage{amscd}
\usepackage[hidelinks]{hyperref}
\usepackage{color}
\usepackage{float}
\usepackage{graphics,amsmath,amssymb}
\usepackage{amsthm}
\usepackage{amsfonts}
\usepackage{latexsym}
\usepackage{epsf}
\usepackage{enumerate}
\usepackage{xifthen}
\usepackage{mathrsfs}
\usepackage{dsfont}
\usepackage{makecell}
\usepackage[FIGTOPCAP]{subfigure}
\usepackage{amsmath}
\allowdisplaybreaks[4]
\usepackage{listings}
\usepackage{etoolbox}
\usepackage{fancyhdr}
\usepackage{pdflscape}
\usepackage[title,toc,titletoc]{appendix}

 \newtheoremstyle{mytheorem}
 {3pt}
 {3pt}
 {\slshape}
 {}
 {\bfseries}
 {.}
 { }
 {}

\numberwithin{equation}{section}

\theoremstyle{theorem}
\newtheorem{theorem}{Theorem}[section]

\newtheorem{lemma}[theorem]{Lemma}

\newtheorem*{theorem*}{Theorem}

\theoremstyle{definition}

\newtheorem{remark}{Remark}[section]

\newcommand{\arxiv}[1]{\href{http://arxiv.org/abs/#1}{arXiv:#1}}

\newcommand{\Keywords}[1]{\ifthenelse{\isempty{#1}}{}{\smallskip \smallskip \noindent \textbf{Keywords}. #1}}
\newcommand{\MSC}[2][2010]{\ifthenelse{\isempty{#2}}{}{\smallskip \smallskip \noindent \textbf{#1MSC}. #2}}
\newcommand{\abstractnote}[1]{\ifthenelse{\isempty{#1}}{}{\smallskip \smallskip \noindent \textsuperscript{\dag}#1}}

\makeatletter
\def\specialsection{\@startsection{section}{1}%
  \z@{\linespacing\@plus\linespacing}{.5\linespacing}%
  {\normalfont}}
\def\section{\@startsection{section}{1}%
  \z@{.7\linespacing\@plus\linespacing}{.5\linespacing}%
  {\normalfont\scshape}}
\patchcmd{\@settitle}{\uppercasenonmath\@title}{\Large\boldmath}{}{}
\patchcmd{\@settitle}{\begin{center}}{\begin{flushleft}}{}{}
\patchcmd{\@settitle}{\end{center}}{\end{flushleft}}{}{}
\patchcmd{\@setauthors}{\MakeUppercase}{\normalsize}{}{}
\patchcmd{\@setauthors}{\centering}{\raggedright}{}{}
\patchcmd{\section}{\scshape}{\large\bfseries\boldmath}{}{}
\patchcmd{\subsection}{\bfseries}{\bfseries\boldmath}{}{}
\renewcommand{\@secnumfont}{\bfseries}
\patchcmd{\@startsection}{\@afterindenttrue}{\@afterindentfalse}{}{}
\patchcmd{\abstract}{\leftmargin3pc}{\leftmargin1pc}{}{}

\def\maketitle{\par
  \@topnum\z@ 
  \@setcopyright
  \thispagestyle{empty}
  \ifx\@empty\shortauthors \let\shortauthors\shorttitle
  \else \andify\shortauthors
  \fi
  \@maketitle@hook
  \begingroup
  \@maketitle
  \toks@\@xp{\shortauthors}\@temptokena\@xp{\shorttitle}%
  \toks4{\def\\{ \ignorespaces}}
  \edef\@tempa{%
    \@nx\markboth{\the\toks4
      \@nx\MakeUppercase{\the\toks@}}{\the\@temptokena}}%
  \@tempa
  \endgroup
  \c@footnote\z@
  \@cleartopmattertags
}
\makeatother

\newcommand{\f}[1]{\ifthenelse{\equal{#1}{1}}{(q;q)_\infty}{(q^{#1};q^{#1})_{\infty}}}
\newcommand{\ee}[1]{e\left(#1\right)}
\newcommand{\bfm}{\mathbf{m}}
\newcommand{\bfd}{\boldsymbol{\delta}}
\newcommand{\cl}{\mathcal{L}}
\newcommand{\lcm}{\mathrm{lcm}}

\title[Asymptotics for eta-quotients]{Asymptotics for the Fourier coefficients of eta-quotients}

\author[S. Chern]{Shane Chern}
\address[S. Chern]{Department of Mathematics, The Pennsylvania State University, University Park, PA 16802, USA}
\email{shanechern@psu.edu}

\date{}

\begin{document}

%

\maketitle

\begin{abstract}

We study the asymptotics for the Fourier coefficients of a broad class of eta-quotients,
$$\prod_{r=1}^R \left(\prod_{k\ge 1}\left(1-q^{m_r k}\right)\right)^{\delta_r},$$
where $m_1,\ldots,m_R$ are $R$ distinct positive integers and $\delta_1,\ldots,\delta_R$ are $R$ non-zero integers with $\sum_{r=1}^R \delta_{r}\ge 0$.


\MSC{11P55, 11P82.}
\end{abstract}

\section{Introduction}

\subsection{Background}

Since the pioneering work of Hardy and Ramanujan \cite{HR1916}, and Rademacher \cite{Rad1937,Rad1943}, there has been an intensive study of the asymptotic behaviors of the Fourier coefficients of eta-quotients. Here what Hardy and Ramanujan as well as Rademacher treated has important combinatorial background:

A partition of a positive integer $n$ is a non-increasing sequence of positive integers whose sum equals $n$. Let $p(n)$ count the number of partitions of $n$ with the convention that $p(0)=1$. It is well known that
\begin{equation}\label{eq:part-gf}
\sum_{n\ge 0}p(n)q^n=\frac{1}{(q;q)_\infty}.
\end{equation}
Here and in the sequel, we adopt the standard $q$-series notation:
\begin{align*}
(a;q)_\infty&:=\prod_{k\ge 0} (1-aq^{k}).
\end{align*}
It is necessary to mention that the generating function \eqref{eq:part-gf} is holomorphic on the open unit disk $\mathbb{D}\subset \mathbb{C}$, namely, $|q|<1$.

Rademacher showed that
\begin{align}
p(n)=\frac{1}{2\sqrt{2} \pi}\sum_{k\ge 1}A_k(n)\sqrt{k}\;\frac{d}{dn}\left(\frac{2}{\sqrt{n-\frac{1}{24}}}\sinh\left(\frac{\pi}{k}\sqrt{\frac{2}{3}\left(n-\frac{1}{24}\right)}\right)\right),
\end{align}
where
$$A_k(n)=\sum_{\substack{0\le h< k\\ (h,k)=1}}e^{\pi i(s(h,k)-2nh/k)}$$
with $s(h,k)$ being the Dedekind sum defined in \eqref{eq:dedekind-sum}.

Rademacher's approach is straightforward in essence, however delicate in detail. The basic idea is merely Cauchy's integral formula, but we need various techniques including Ford circles, Farey sequences, modular symmetry and the Dedekind eta-function.

One may naturally consider the following general family of holomorphic functions on the open unit disk $\mathbb{D}$:
\begin{equation}\label{eq:eta-quo}
G(q)=G(e^{2\pi i \tau}):=\prod_{r=1}^R (q^{m_r};q^{m_r})_{\infty}^{\delta_r},
\end{equation}
where $\bfm=(m_1,\ldots,m_R)$ is a sequence of $R$ distinct positive integers and $\bfd=(\delta_1,\ldots,\delta_R)$ is a sequence of $R$ non-zero integers. This can be treated as an eta-quotient as the Dedekind eta-function is defined by (with $q=e^{2\pi i \tau}$)
$$\eta(\tau):=q^{\frac{1}{24}}\f{1}.$$

If we write
\begin{equation}
G(q)=\sum_{n\ge 0} g(n) q^n,
\end{equation}
then $g(n)$ may be related to certain combinatorial quantity in the theory of partitions. We notice that Rademacher's approach for the asymptotic formula of $p(n)$ can be easily adapted to study the asymptotics of $g(n)$ provided $\sum_{r=1}^R \delta_{r}<0$. For some particular $G(q)$, the interested readers may refer to the work of Grosswald \cite{Gro1958}, Hagis Jr.~\cite{Hag1963,Hag1965,Hag1966}, Iseki \cite{Ise1960,Ise1961} and many others. On the other hand, when $\sum_{r=1}^R \delta_{r}\ge 0$, a variance of Rademacher's method provided by O-Y.~Chan \cite{Cha2005} works. Chan used this method to prove several conjectures of Andrews and Lewis \cite{AL2000} concerning rank/crank differences modulo $3$ and $4$. We refer the readers to \cite{CTW2018,KN2014,Kim2018} for other related papers.

For general $G(q)$ with $\sum_{r=1}^R \delta_{r}<0$, the recent work of Sussman \cite{Sus2017} presented a Rademacher-type formula. Sussman's result can in some sense be treated as a special case of the work of Bringmann and Ono \cite{BO2012}, in which the coefficients of harmonic Maass forms are studied. On the other hand, Sills \cite{Sil2010} provided an automatic algorithm when $\sum_{r=1}^R \delta_{r}=0$. When $\sum_{r=1}^R \delta_{r}=1$, B.~Kim \cite{Kim2012} studied a subclass of such $G(q)$ with the aid of Chan's method.

The purpose of this paper is to obtain asymptotic formulas of Fourier coefficients of a broad class of eta-quotients $G(q)$ with $\sum_{r=1}^R \delta_{r}\ge 0$.

\subsection{The main result and some remarks}

Assuming that $k$ and $h$ are positive integers with $\gcd(h,k)=1$, we put
\begin{equation}
\begin{aligned}
\Delta_1 = -\frac{1}{2}\sum_{r=1}^R \delta_{r},& \qquad \Delta_2 = \sum_{r=1}^R m_r \delta_r,\\
\Delta_3(k) = -\sum_{r=1}^R \frac{\delta_r \gcd^2(m_r,k)}{m_r},& \qquad \Delta_4(k) = \prod_{r=1}^R \left(\frac{m_r}{\gcd(m_r,k)}\right)^{-\frac{\delta_r}{2}},
\end{aligned}
\end{equation}
and
\begin{align}
\omega_{h,k}=\exp\left(-\pi i\sum_{r=1}^R \delta_r\; s\left(\frac{m_r h}{\gcd(m_r,k)},\frac{k}{\gcd(m_r, k)}\right)\right),
\end{align}
where $s(d,c)$ is the Dedekind sum defined by
\begin{equation}\label{eq:dedekind-sum}
s(d,c):=\sum_{n \bmod{c}} \bigg(\bigg(\frac{dn}{c}\bigg)\bigg)\bigg(\bigg(\frac{n}{c}\bigg)\bigg)
\end{equation}
with
$$((x)):=\begin{cases}
x-\lfloor x\rfloor -1/2 & \text{if $x\not\in \mathbb{Z}$},\\
0 & \text{if $x\in \mathbb{Z}$}.
\end{cases}$$

Let $L=\lcm(m_1,\ldots,m_R)$. We divide the set $\{1,2,\ldots,L\}$ into two disjoint subsets:
\begin{align*}
\cl_{>0}&:=\{1\le \ell \le L \; :\; \Delta_3(\ell)>0\},\\
\cl_{\le 0}&:=\{1\le \ell \le L \; :\; \Delta_3(\ell)\le0\}.
\end{align*}

Now we are ready to present the main result.

\begin{theorem}\label{th:main}
If $\Delta_1\le 0$ and the inequality
\begin{equation}\label{eq:assump}
\min_{1\le r\le R}\left(\frac{\gcd^2(m_r,\ell)}{m_r}\right)\ge \frac{\Delta_3(\ell)}{24}
\end{equation}
holds for all $1\le \ell \le L$, then for positive integers $n>-\Delta_2/24$, we have
\begin{align}
g(n)&=E(n)+\sum_{\ell\in \cl_{>0}} 2\pi\;\Delta_4(\ell) \left(\frac{24n+\Delta_2}{\Delta_3(\ell)}\right)^{-\frac{\Delta_1+1}{2}}\nonumber\\
&\quad\quad\quad\quad\quad \times\sum_{\substack{1\le k\le N\\k\equiv_L \ell}}\frac{I_{-\Delta_1-1}\left(\frac{\pi }{6k}\sqrt{\Delta_3(\ell)(24n+\Delta_2)}\right)}{k} \sum_{\substack{0\le h< k\\ (h,k)=1}} \omega_{h,k}\; \ee{-\frac{nh}{k}},
\end{align}
where
\begin{equation}
N=\left\lfloor \sqrt{2\pi \left(n+\frac{\Delta_2}{24}\right)}\right\rfloor,
\end{equation}
\begin{align}
|E(n)|\ll_{\bfm,\bfd}\begin{cases}
1 & \text{if $\Delta_1=0$},\\
\left(n+\frac{\Delta_2}{24}\right)^{1/4} & \text{if $\Delta_1=-\frac{1}{2}$},\\
\left(n+\frac{\Delta_2}{24}\right)^{1/2} \log \left(n+\frac{\Delta_2}{24}\right)& \text{if $\Delta_1=-1$},\\
\left(n+\frac{\Delta_2}{24}\right)^{-\Delta_1-1/2} & \text{if $\Delta_1\le -\frac{3}{2}$},\\
\end{cases}
\end{align}
and $I_s(x)$ is the modified Bessel function of the first kind.
\end{theorem}

We have several remarks to make.

\begin{enumerate}[\noindent 1.]
\item The main term of $g(n)$ may vanish for certain families of eta-quotients. One example arises from Gauss' square exponent theorem \cite[Corollary 2.10]{And1976}:
\begin{equation}\label{eq:ga}
\frac{\f{1}^2}{\f{2}}=\sum_{k=-\infty}^\infty (-1)^k q^{k^2},
\end{equation}
from which we see that the Fourier coefficients of the left-hand side are in $\{0,1,\pm 2\}$. In this case, we have $L=2$, and hence we can compute
$$\Delta_3(1)=-\frac{3}{2}\qquad \text{and} \qquad \Delta_3(2)=0.$$
This implies that the set $\cl_{>0}$ is empty. Consequently, the main term vanishes for all $n$.

\item Recalling that the asymptotic expansion of $I_s(x)$ (cf.~\cite[p.~377, (9.7.1)]{AA1972}) states that, for fixed $s$, when $|\arg x|<\frac{\pi}{2}$,
\begin{align}\label{Bessel-order}
I_{s}(x)\sim \frac{e^x}{\sqrt{2\pi x}}\left(1-\frac{4s^2-1}{8x}+\frac{(4s^2-1)(4s^2-9)}{2!(8x)^2}-\cdots \right),
\end{align}
one may say more about the main term. Let $\ell\in\cl_{>0}$ be fixed. One thing we observe is that for given $k$, the sum
\begin{equation*}
\sum_{\substack{0\le h< k\\ (h,k)=1}} \omega_{h,k}\; \ee{-\frac{nh}{k}}
\end{equation*}
has period $k$ for $n\in\mathbb{Z}_{>0}$. On the other hand, letting us fix some $k_0\equiv \ell \pmod{L}$, the following estimate of the tail gives useful information:
\begin{align}
\mathrm{Tail}:=&\sum_{\substack{k_0< k\le N\\k\equiv_L \ell}}\frac{I_{-\Delta_1-1}\left(\frac{\pi }{6k}\sqrt{\Delta_3(\ell)(24n+\Delta_2)}\right)}{k} \sum_{\substack{0\le h< k\\ (h,k)=1}} \omega_{h,k}\; \ee{-\frac{nh}{k}} \nonumber\\
\ll&\sum_{\substack{k_0< k\le N\\k\equiv_L \ell}}I_{-\Delta_1-1}\left(\frac{\pi }{6k}\sqrt{\Delta_3(\ell)(24n+\Delta_2)}\right) \nonumber\\
\ll&\sum_{\substack{k_0< k\le N\\k\equiv_L \ell}} \frac{e^{\frac{\pi }{6k}\sqrt{\Delta_3(\ell)(24n+\Delta_2)}}}{k^{-\frac{1}{2}}(24n+\Delta_2)^{\frac{1}{4}}} \nonumber\\
\ll&\;N^{\frac{3}{2}}(24n+\Delta_2)^{-\frac{1}{4}}e^{\frac{\pi }{6(k_0+L)}\sqrt{\Delta_3(\ell)(24n+\Delta_2)}} \nonumber\\
\ll&\;(24n+\Delta_2)^{\frac{1}{2}}e^{\frac{\pi }{6(k_0+L)}\sqrt{\Delta_3(\ell)(24n+\Delta_2)}} \nonumber\\
=&\;o\left(e^{\frac{\pi }{6k_0}\sqrt{\Delta_3(\ell)(24n+\Delta_2)}}\right).\label{eq:tail-est}
\end{align}

\item An explicit bound of $E(n)$ is given in \eqref{eq:exp-E-bound}. It would be necessary if one wants to study the exact sign pattern of $g(n)$.

\item It is helpful to record Sussman's result \cite{Sus2017}:
\begin{theorem*}[Sussman]
If $\Delta_1> 0$ and the inequality
\begin{equation*}
\min_{1\le r\le R}\left(\frac{\gcd^2(m_r,\ell)}{m_r}\right)\ge \frac{\Delta_3(\ell)}{24}
\end{equation*}
holds for all $1\le \ell \le L$, then for positive integers $n>-\Delta_2/24$, we have
\begin{align}
g(n)&=\sum_{\ell\in \cl_{>0}} 2\pi\;\Delta_4(\ell) \left(\frac{24n+\Delta_2}{\Delta_3(\ell)}\right)^{-\frac{\Delta_1+1}{2}}\nonumber\\
&\quad\quad\quad\quad\ \  \times\sum_{\substack{k\ge 1\\k\equiv_L \ell}}\frac{I_{\Delta_1+1}\left(\frac{\pi }{6k}\sqrt{\Delta_3(\ell)(24n+\Delta_2)}\right)}{k} \sum_{\substack{0\le h< k\\ (h,k)=1}} \omega_{h,k}\; \ee{-\frac{nh}{k}}.
\end{align}
\end{theorem*}
\noindent One may notice that Sussman's asymptotic formula is almost the same as our main term. His proof is mainly based on Rademacher's original version \cite{Rad1937} of circle method, whereas ours relies on O-Y.~Chan's version \cite{Cha2005}. One major difference arises from the contours chosen for Cauchy's integral formula. 
\end{enumerate}

\subsection{Notation}

Let $\mathbb{Z}$, $\mathbb{R}$ and $\mathbb{C}$ be the set of integers, real numbers and complex numbers, respectively. Let $\mathbb{H}$ be the upper half complex plane. For a given set $\mathcal{S}$, we use $\mathcal{S}_{\text{condition}}$ to denote the subset of $\mathcal{S}$ with the given condition (e.g.~$\mathbb{Z}_{>0}$ denotes the set of positive integers).

The big-$O$ notation is defined in the usual way: $f(x)=O(g(x))$ means that $|f(x)|\le C g(x)$ where $C$ is an absolute constant. If the constant $C$ depends on some variables, then we write $f(x)=O_{\text{variables}}(g(x))$. We may also omit the subscript if these variables are explicit. Furthermore, $f(x)\ll g(x)$ means that $f(x)=O(g(x))$. On the other hand, if $\lim f(x)/g(x)=0$, then we say $f(x)=o(g(x))$. If $\lim f(x)/g(x)=1$, then we write $f(x)\sim g(x)$.

We use $u\equiv_m v $ to denote $u\equiv v \pmod{m}$, and likewise $u\not\equiv_m v $ to denote $u\not\equiv v \pmod{m}$.

For $\Delta\in \frac{1}{2}\mathbb{Z}_{\le 0}$ and $x\in\mathbb{R}_{\ge 1}$, we define
\begin{align}\label{eq:xi-rep}
\Xi_{\Delta}(x):=\begin{cases}
1 & \text{if $\Delta=0$},\\
2x^{1/2} & \text{if $\Delta=-\frac{1}{2}$},\\
x(\log x+1) & \text{if $\Delta=-1$},\\
\zeta(-\Delta)x^{-2\Delta-1} & \text{otherwise},\\
\end{cases}
\end{align}
where $\zeta(\cdot)$ is the Riemann zeta-function.

At last, we use the conventional notation $e(z):=e^{2 \pi i z}$ for complex $z$.

\section{The circle method}

Our main ingredient is O-Y.~Chan's version of circle method \cite{Cha2005} with some slight modifications. For the sake of completeness, we shall reproduce the necessary details.

\subsection{Cauchy's integral formula}

Given a holomorphic function
$$G(q)=\sum_{n\ge 0} g(n) q^n$$
on a simply connected domain containing the origin, we know from Cauchy's integral formula that
\begin{equation*}
g(n)=\frac{1}{2 \pi i} \oint_{|q|=r} \frac{G(q)}{q^{n+1}}\ dq,
\end{equation*}
where the contour integral is taken counter-clockwise.

We take $r=e^{-2 \pi \varrho}$ with $\varrho=1/N^2$ where $N$ is a positive integer to be determined in the sequel.

Let $h/k$ be a Farey fraction of order $N$. We use $\xi_{h,k}$ to denote the interval $[-\theta'_{h,k},\theta''_{h,k}]$ with $-\theta'_{h,k}$ and $\theta''_{h,k}$ being the positive distances from $h/k$ to its neighboring mediants. Now we dissect the circle $|q|=r$ by Farey arcs and obtain
\begin{equation*}
g(n)=\sum_{1\le k\le N} \sum_{\substack{0\le h< k\\ (h,k)=1}} \ee{-\frac{nh}{k}} \int_{\xi_{h,k}} G\big(\ee{h/k+i \varrho +\phi}\big)\ee{-n \phi} e^{2 \pi n \varrho}\ d\phi.
\end{equation*}

Letting $z=k(\varrho -i \phi)$ and $\tau = (h+i z)/k$, we may rewrite the integral as
\begin{equation}\label{eq:cauchy-var}
g(n)=\sum_{1\le k\le N} \sum_{\substack{0\le h< k\\ (h,k)=1}} \ee{-\frac{nh}{k}} \int_{\xi_{h,k}} G\big(\ee{\tau}\big)\ee{-n \phi} e^{2 \pi n \varrho}\ d\phi.
\end{equation}

\subsection{Transformation formulas}

Let
$$F(q)=F\big(\ee{\tau}\big):=\frac{1}{\f{1}}$$
with $\tau\in\mathbb{H}$. We also define the fractional linear transformation
$$\gamma(\tau):=\frac{a\tau +b}{c\tau + d}, \quad \text{with }\gamma=\begin{pmatrix}a & b\\c& d\end{pmatrix}\in SL_2(\mathbb{Z}).$$

We know from the transformation formula of Dedekind eta-function \cite[pp.~52--61]{Apo1990} that
\begin{equation}\label{eq:F-gamma}
F\big(\ee{\tau}\big)=\ee{\frac{\tau-\gamma(\tau)}{24}-\frac{s(d,c)}{2}+\frac{a+d}{24c}} \sqrt{-i (c\tau +d)}\ F\big(\ee{\gamma(\tau)}\big),
\end{equation}
where the square root is taken on the principal branch, with $z^{1/2}>0$ for $z>0$.

Given a positive integer $m$, to transform $F(e(m\tau))$ properly for each choice of $k$ and $h$ (with $\gcd(h,k)=1$), we proceed as follows.

Let $d=\gcd(m,k)$. We also write $m=d m'$ and $k= d k'$. Choosing an integer $\tilde{h}_{m'}$ such that $\tilde{h}_{m'} m' h\equiv -1 \pmod{k'}$ (this is possible since $\gcd(h,k)=1$) and putting $b_{m'}=(\tilde{h}_{m'} m' h+1)/k'$, we obtain the following matrix in $SL_2(\mathbb{Z})$:
\begin{equation}
\gamma_{(m,k)}=\begin{pmatrix}\tilde{h}_{m'} & -b_{m'}\\k' & -m' h\end{pmatrix}.
\end{equation}
Note that our $\gamma_{(m,k)}$ is different to Chan's definition. However, they are identical if $\gcd(m,k)=1$.

Recall that
$$\tau=\frac{h+iz}{k}=\frac{h+iz}{d k'}.$$
We have
\begin{align*}
\gamma_{(m,k)}(m\tau)&=\dfrac{\tilde{h}_{m'}\cdot m\frac{h+iz}{d k'}-b_{m'}}{k' \cdot m\frac{h+iz}{d k'}-m' h}=\frac{\tilde{h}_{m'} m' h+\tilde{h}_{m'}(im'z)-(\tilde{h}_{m'}m' h+1)}{m' h k'+k'(im'z)-m' hk'}\\
&=\frac{\tilde{h}_{m'}}{k'}+\frac{1}{m'k'z}i.
\end{align*}
This implies that
\begin{equation}\label{eq:gamma}
\gamma_{(m,k)}(m\tau)=\frac{\tilde{h}_{m'} \gcd(m,k)}{k}+\frac{\gcd^2(m,k)z^{-1}}{mk}i.
\end{equation}

On the other hand, since $\gamma_{(m,k)}\in SL_2(\mathbb{Z})$, it follows from \eqref{eq:F-gamma} that
\begin{align}
F\big(\ee{m\tau}\big)&=e^{\frac{\pi}{12 k}\left(\frac{\gcd^2(m,k)}{mz}-mz\right)} \ee{\frac{s\left(\frac{m h}{\gcd(m,k)},\frac{k}{\gcd(m, k)}\right)}{2}}\nonumber\\
&\qquad\times \sqrt{\frac{m z}{\gcd(m,k)}}\ F\big(\ee{\gamma_{(m,k)}(m\tau)}\big). \label{eq:F-m-gamma}
\end{align}

Let $G(q)$ be defined by \eqref{eq:eta-quo}. Then
\begin{equation}
G\left(\ee{\tau}\right)=\prod_{r=1}^R F\left(\ee{m_r \tau}\right)^{-\delta_{r}}.
\end{equation}
The following result is an immediate consequence of \eqref{eq:F-m-gamma}.

\begin{lemma}
We have, for any positive integers $k$ and $h$ with $\gcd(h,k)=1$,
\begin{align}\label{eq:G-m-gamma}
G\big(\ee{\tau}\big)&=e^{\frac{\pi}{12 k}\left(\frac{\Delta_3(k)}{z}+\Delta_2 z\right)}\; z^{\Delta_1}\; \omega_{h,k}\; \Delta_4(k)\; \prod_{r=1}^R F\big(\ee{\gamma_{(m_r,k)}(m_r\tau)}\big)^{-\delta_{r}}.
\end{align}
\end{lemma}

\subsection{Some bounds}

We now present some bounds, most of which are given by Chan.

A standard result on Farey fractions states that (cf.~\cite[Chap.~3]{HW1979})
\begin{equation}\label{eq:theta-bound}
\frac{1}{2kN}\le \theta'_{h,k},\theta''_{h,k}\le \frac{1}{kN}.
\end{equation}
Consequently,
\begin{equation}\label{eq:xi-bound}
\frac{1}{kN}\le |\xi_{h,k}| \le \frac{2}{kN}.
\end{equation}

We next notice that $z=k(\varrho -i \phi)$. Hence
\begin{equation}\label{eq:Re-z-bound}
\Re(z)=k\varrho=\frac{k}{N^2}.
\end{equation}
This implies that
\begin{equation}\label{eq:z-bound}
|z|\ge \frac{k}{N^2}.
\end{equation}

On the other hand, we have
\begin{align}\label{eq:Re-1-z-bound}
\Re\left(\frac{1}{z}\right)\ge \frac{k}{2}.
\end{align}
This is because
\begin{align*}
\Re\left(\frac{1}{z}\right)=\frac{1}{k}\frac{\varrho}{\varrho^2+\phi^2}\ge \frac{1}{k}\frac{N^{-2}}{N^{-4} + k^{-2} N^{-2}}=\frac{k}{k^2N^{-2}+1}\ge \frac{k}{1+1}=\frac{k}{2},
\end{align*}
where we use the fact $k\le N$ in the last inequality.

In the author's recent work with Tang and Wang \cite{CTW2018}, it is shown that

\begin{lemma}
We have
\begin{align}
\left|F\big(\ee{\gamma_{(m,k)}(m\tau)}\big)\right|&\le \exp\left(\frac{e^{-\pi \gcd^2(m,k)/m}}{\left(1-e^{-\pi \gcd^2(m,k)/m}\right)^2}\right),\label{eq:bound-1}\\
\left|\frac{1}{F\big(\ee{\gamma_{(m,k)}(m\tau)}\big)}\right|&\le \exp\left(\frac{e^{-\pi \gcd^2(m,k)/m}}{\left(1-e^{-\pi \gcd^2(m,k)/m}\right)^2}\right).\label{eq:bound-2}
\end{align}
\end{lemma}

\begin{proof}
We write in this proof $q=\ee{y}$ with $y\in \mathbb{H}$. Recall that $p(n)$ denotes the number of partitions of $n$ with the convention that $p(0)=1$. It is clear that $p(n)>0$ for all $n\ge 0$.

We first have
$$|F(q)|\le \sum_{n\ge 0} p(n) |q|^n = F(|q|).$$

On the other hand,
\begin{equation*}
\left|\frac{1}{F(q)}\right|=\prod_{k\ge 1} \left|1-q^k\right|\le \prod_{k\ge 1} \left(1+|q|^k\right)\le \prod_{k\ge 1} \frac{1}{1-|q|^k}=F(|q|).
\end{equation*}

For real $0\le x< 1$, we have
\begin{align*}
\log F(x)&=-\sum_{k\ge 1}\log(1-x^k)=\sum_{k\ge 1}\sum_{m\ge 1}\frac{x^{km}}{m}=\sum_{n\ge 1} x^n \sum_{d\mid n}\frac{1}{d} \\
& \le \sum_{n\ge 1} n x^n = \frac{x}{(1-x)^2}.
\end{align*}
Noting that $|q|=|e(y)|=e^{-2\pi \Im(y)}$, we have
\begin{equation}\label{eq:Fq-bound}
F(|q|)\le \exp\left(\frac{e^{-2\pi \Im(y)}}{(1-e^{-2\pi \Im(y)})^2}\right).
\end{equation}
We remark that the right-hand side can be treated as a decreasing function of $\Im(y)$.

At last, we see from \eqref{eq:gamma} that
$$\Im\left(\gamma_{(m,k)}(m\tau)\right)=\frac{\gcd^2(m,k)}{mk}\Re\left(\frac{1}{z}\right)\ge \frac{\gcd^2(m,k)}{mk}\frac{k}{2}=\frac{\gcd^2(m,k)}{2m},$$
where we use \eqref{eq:Re-1-z-bound}. Hence the left-hand sides of \eqref{eq:bound-1} and \eqref{eq:bound-2} are both
$$\le \exp\left(\frac{e^{-2\pi \Im\left(\ee{\gamma_{(m,k)}(m\tau)}\right)}}{\left(1-e^{-2\pi \Im\left(\ee{\gamma_{(m,k)}(m\tau)}\right)}\right)^2}\right)\le \exp\left(\frac{e^{-\pi \gcd^2(m,k)/m}}{\left(1-e^{-\pi \gcd^2(m,k)/m}\right)^2}\right).$$
\end{proof}

We also extend a bound given in \cite[Lemma 3.3]{CTW2018}.

\begin{lemma}\label{le:rough-bound}
Let $\eta_1,\ldots,\eta_R$ be $R$ integers, all non-zero, and let $y_1,\ldots,y_R$ be $R$ complex numbers in $\mathbb{H}$. Then
\begin{equation}
\left|\prod_{r=1}^R F\left(\ee{y_r}\right)^{\eta_r}-1\right|\le \exp \left(\sum_{r=1}^R \frac{|\eta_r|e^{-2\pi \Im(y_r)}}{(1-e^{-2\pi \Im(y_r)})^2}\right)-1.
\end{equation}
\end{lemma}

\begin{proof}
Let $q_r = e(y_r)$ for $r=1,\ldots,R$.

Note that $F(q_r)^{\eta_r}$ can be treated as the generating function of $|\eta_r|$-colored partitions when $\eta_r$ is a positive integer, while when $\eta_r$ is a negative integer, it is the generating function of weighted $|\eta_r|$-colored distinct partitions where the weight is given by $(-1)^\sharp$ with $\sharp$ counting the total number of parts.

Hence, for $r=1,\ldots,R$, if we write
$$F\left(\ee{y_r}\right)^{\eta_r}=\sum_{n\ge 0} a_r(n) q_r^n \quad\text{and}\quad F\left(\ee{y_r}\right)^{|\eta_r|}=\sum_{n\ge 0} \tilde{a}_r(n) q_r^n,$$
then $|a_r(n)|\le \tilde{a}_r(n)$ for $n\ge 0$ with $a_r(0)=\tilde{a}_r(0)=1$.

Hence
\begin{align*}
\left|\prod_{r=1}^R F\left(\ee{y_r}\right)^{\eta_r}-1\right|&=\left|\sum_{n_1\ge 0}\cdots \sum_{n_R\ge 0} a_1(n_1)\cdots a_R(n_R) q_1^{n_1}\cdots q_R^{n_R}-1\right|\\
&=\left|\underset{(n_1,\ldots,n_R)\ne (0,\ldots,0)}{\sum_{n_1\ge 0}\cdots \sum_{n_R\ge 0}} a_1(n_1)\cdots a_R(n_R) q_1^{n_1}\cdots q_R^{n_R}\right|\\
&\le \underset{(n_1,\ldots,n_R)\ne (0,\ldots,0)}{\sum_{n_1\ge 0}\cdots \sum_{n_R\ge 0}} \tilde{a}_1(n_1)\cdots \tilde{a}_R(n_R) |q_1|^{n_1}\cdots |q_R|^{n_R}\\
&=\sum_{n_1\ge 0}\cdots \sum_{n_R\ge 0} \tilde{a}_1(n_1)\cdots \tilde{a}_R(n_R) |q_1|^{n_1}\cdots |q_R|^{n_R} -1\\
&= \prod_{r=1}^R F(|q_r|)^{|\eta_r|}-1\le  \exp \left(\sum_{r=1}^R \frac{|\eta_r|e^{-2\pi \Im(y_r)}}{(1-e^{-2\pi \Im(y_r)})^2}\right)-1,
\end{align*}
where we use \eqref{eq:Fq-bound} in the last inequality.
\end{proof}

\subsection{An integral}

The last task in this section is to evaluate a useful integral.

\begin{lemma}\label{le:key-int}
Let $a\in \mathbb{R}_{>0}$, $b\in\mathbb{R}$ and $c\in\frac{1}{2}\mathbb{Z}_{\le 0}$. Let $\gcd(h,k)=1$. Define
\begin{equation}
I:=\int_{\xi_{h,k}}e^{\frac{\pi}{12k}\left(\frac{a}{z}+bz\right)}z^{c}\ee{-n\phi}e^{2\pi n \varrho}\ d\phi.
\end{equation}
Then, for those positive integers $n$ with $24n+b> 0$, we have
\begin{equation}
I=\frac{2\pi}{k} \left(\frac{24n+b}{a}\right)^{-\frac{c+1}{2}} I_{-c-1}\left(\frac{\pi }{6k}\sqrt{a(24n+b)}\right)+E(I)
\end{equation}
where
\begin{equation}
|E(I)|\le\frac{2^{-c}\pi^{-1}e^{\frac{a\pi}{3}}N^{-c}}{n+\frac{b}{24}}e^{2\pi\varrho \left(n+\frac{b}{24}\right)}.
\end{equation}
\end{lemma}

\begin{proof}
We first put $w=z/k=\varrho-i\phi$ to obtain
$$I=\frac{1}{2\pi i}\int_{\varrho-i\theta''_{h,k}}^{\varrho+i\theta'_{h,k}} 2\pi e^{\frac{a\pi}{12k^2 w}} e^{2\pi w \left(n+\frac{b}{24}\right)}(kw)^{c}\ dw.$$
One may separate the integral into three parts
\begin{align*}
I&=\frac{1}{2\pi i}\left(\int_\Gamma-\int_{-\infty-i\theta''_{h,k}}^{\varrho-i\theta''_{h,k}}+\int_{-\infty+i\theta'_{h,k}}^{\varrho+i\theta'_{h,k}}\right) 2\pi e^{\frac{a\pi}{12k^2 w}} e^{2\pi w \left(n+\frac{b}{24}\right)}(kw)^{c}\ dw\\
&=:J_1-J_2+J_3,
\end{align*}
where
\begin{align*}
\Gamma:=(-\infty-i\theta''_{h,k}) \to (\varrho-i\theta''_{h,k}) \to (\varrho+i\theta'_{h,k}) \to (-\infty+i\theta'_{h,k})
\end{align*}
is a Hankel contour.

To compute the main term $J_1$, we make the following change of variables $t=wk\sqrt{(24n+b)/a}$ to obtain
$$J_1=\frac{2\pi}{k}\left(\frac{24n+b}{a}\right)^{-\frac{c+1}{2}} \frac{1}{2\pi i}\int_{\tilde{\Gamma}} e^{\frac{\pi }{12k}\sqrt{a(24n+b)}\left(t+\frac{1}{t}\right)} t^{c}\  dt.$$
Note that the new contour $\tilde{\Gamma}$ is still a Hankel contour. Recalling the contour integral representation of $I_s(x)$:
$$I_s(x)=\frac{1}{2\pi i}\int_{\Gamma} t^{-s-1}e^{\frac{x}{2}\left(t+\frac{1}{t}\right)}\ dt\quad\text{($\Gamma$ is a Hankel contour)},$$
we conclude
$$J_1=\frac{2\pi}{k} \left(\frac{24n+b}{a}\right)^{-\frac{c+1}{2}} I_{-c-1}\left(\frac{\pi }{6k}\sqrt{a(24n+b)}\right).$$

For the error term $E(I)$, which comes from $J_2$ and $J_3$, we put $w=x+i\theta$ with $-\infty\le x\le \varrho$ and $\theta\in\{\theta'_{h,k},-\theta''_{h,k}\}$. Following Chan \cite[p.~120]{Cha2005}, we have
$$\left|e^{2\pi w \left(n+\frac{b}{24}\right)}\right|= e^{2\pi x \left(n+\frac{b}{24}\right)},\\$$
and
\begin{align*}
\left|e^{\frac{a\pi}{12k^2 w}}\right|&=e^{\frac{a\pi}{12k^2}\Re\left(\frac{1}{w}\right)}=e^{\frac{a\pi}{12k^2}\frac{x}{x^2+\theta^2}}\le e^{\frac{a\pi}{12k^2}\frac{x}{\theta^2}}\le e^{\frac{a\pi}{12k^2}\varrho (2kN)^2}=e^{\frac{a\pi}{3}},\\
\left|(kw)^c\right|&=\left(|kw|^{-1}\right)^{-c}\le \left(\frac{1}{k\sqrt{x^2+\theta^2}}\right)^{-c}\le \left(\frac{1}{k|\theta|}\right)^{-c}\le (2N)^{-c},
\end{align*}
where we use the bound $\frac{1}{2kN}\le |\theta|\le \frac{1}{kN}$. Hence for $j=2$ and $3$, we have
\begin{align*}
|J_j|&\le \frac{1}{2\pi} \int_{-\infty}^{\varrho} 2\pi e^{\frac{a\pi}{3}}e^{2\pi x \left(n+\frac{b}{24}\right)}\ (2N)^{-c}\ dx\\
&=\frac{2^{-c-1}\pi^{-1}e^{\frac{a\pi}{3}}N^{-c}}{n+\frac{b}{24}}e^{2\pi\varrho \left(n+\frac{b}{24}\right)}.
\end{align*}
This implies that
$$|E(I)|=|-J_2+J_3|\le |J_2|+|J_3|\le \frac{2^{-c}\pi^{-1}e^{\frac{a\pi}{3}}N^{-c}}{n+\frac{b}{24}}e^{2\pi\varrho \left(n+\frac{b}{24}\right)}.$$
\end{proof}

\section{Asymptotics}

We rewrite \eqref{eq:cauchy-var} as follows:

\begin{align}
g(n)&=\sum_{1\le k\le N} \sum_{\substack{0\le h< k\\ (h,k)=1}} \ee{-\frac{nh}{k}} \int_{\xi_{h,k}} G\big(\ee{\tau}\big)\ee{-n \phi} e^{2 \pi n \varrho}\ d\phi \nonumber\\
&=\sum_{\ell=1}^L\sum_{\substack{1\le k\le N\\k\equiv_L \ell}} \sum_{\substack{0\le h< k\\ (h,k)=1}} \ee{-\frac{nh}{k}} \int_{\xi_{h,k}} G\big(\ee{\tau}\big)\ee{-n \phi} e^{2 \pi n \varrho}\ d\phi \nonumber\\
&=:\sum_{\ell=1}^L S_{\ell}.
\end{align}

Recalling that $L=\lcm(m_1,\ldots,m_R)$, we have, for those $k$ with $k\equiv \ell \pmod{L}$, the following two identities hold:
\begin{align*}
\Delta_3(k)=\Delta_3(\ell)\qquad\text{and}\qquad \Delta_4(k)=\Delta_4(\ell),
\end{align*}
since $\gcd(m_r,k)=\gcd(m_r,\ell)$ for all $r=1,\ldots,R$.

Throughout this section, we always assume that $24n+\Delta_2>0$.

\subsection{Estimating $S_{\ell}$ with $\ell\in\cl_{\le 0}$}

We first assume $\ell\in\cl_{\le 0}$. In this case, the contribution of $S_{\ell}$ is relatively small.

It follows from \eqref{eq:G-m-gamma} and the bounds \eqref{eq:Re-z-bound}--\eqref{eq:bound-2} that, if $k\equiv \ell \pmod{L}$, then
\begin{align*}
&\left|G\big(\ee{\tau}\big)\ee{-n \phi} e^{2 \pi n \varrho}\right|\\
&\quad= \left|\prod_{r=1}^R F\big(\ee{\gamma_{(m_r,k)}(m_r\tau)}\big)^{-\delta_{r}}\right| \left|e^{\frac{\pi}{12 k}\left(\frac{\Delta_3(\ell)}{z}+\Delta_2 z\right)}\right| \Delta_4(\ell) |z|^{\Delta_1} e^{2\pi n \varrho}\\
&\quad= \left|\prod_{r=1}^R F\big(\ee{\gamma_{(m_r,k)}(m_r\tau)}\big)^{-\delta_{r}}\right| e^{\frac{\pi}{12 k}\left(\Delta_3(\ell)\Re\left(\frac{1}{z}\right)+\Delta_2 \Re(z)\right)} \Delta_4(\ell) |z|^{\Delta_1} e^{2\pi n \varrho}\\
&\quad \le \Delta_4(\ell)\; k^{\Delta_1} N^{-2\Delta_1} e^{2\pi\varrho \left(n+\frac{\Delta_2}{24}\right)} \exp\left(\frac{\pi\Delta_3(\ell)}{24}+\sum_{r=1}^R\frac{|\delta_r| e^{-\pi \gcd^2(m_r,\ell)/m_r}}{\left(1-e^{-\pi \gcd^2(m_r,\ell)/m_r}\right)^2}\right).
\end{align*}

Consequently,
\begin{align*}
|S_\ell|&=\left|\sum_{\substack{1\le k\le N\\k\equiv_L \ell}}\sum_{\substack{0\le h< k\\ (h,k)=1}} \ee{-\frac{nh}{k}} \int_{\xi_{h,k}} G\big(\ee{\tau}\big)\ee{-n \phi} e^{2 \pi n \varrho}\ d\phi\right|\\
& \le \sum_{\substack{1\le k\le N\\k\equiv_L \ell}}\sum_{\substack{0\le h< k\\ (h,k)=1}} \int_{\xi_{h,k}} \Delta_4(\ell)\; k^{\Delta_1} N^{-2\Delta_1} e^{2\pi\varrho \left(n+\frac{\Delta_2}{24}\right)}\\
&\quad\quad\quad\quad\quad\quad\quad\quad\times \exp\left(\frac{\pi\Delta_3(\ell)}{24}+\sum_{r=1}^R\frac{|\delta_r| e^{-\pi \gcd^2(m_r,\ell)/m_r}}{\left(1-e^{-\pi \gcd^2(m_r,\ell)/m_r}\right)^2}\right) \ d\phi\\
&\le \sum_{\substack{1\le k\le N\\k\equiv_L \ell}}\sum_{\substack{0\le h< k\\ (h,k)=1}} \Delta_4(\ell)\; k^{\Delta_1} N^{-2\Delta_1} e^{2\pi\varrho \left(n+\frac{\Delta_2}{24}\right)}\\
&\quad\quad\quad\quad\quad\quad\quad\quad\times \exp\left(\frac{\pi\Delta_3(\ell)}{24}+\sum_{r=1}^R\frac{|\delta_r| e^{-\pi \gcd^2(m_r,\ell)/m_r}}{\left(1-e^{-\pi \gcd^2(m_r,\ell)/m_r}\right)^2}\right)\frac{2}{kN} \tag{by \eqref{eq:xi-bound}}\\
&\le \sum_{\substack{1\le k\le N\\k\equiv_L \ell}} \Delta_4(\ell)\; k^{\Delta_1} N^{-2\Delta_1} e^{2\pi\varrho \left(n+\frac{\Delta_2}{24}\right)}\\
&\quad\quad\quad\quad\quad\quad\quad\quad\times  \exp\left(\frac{\pi\Delta_3(\ell)}{24}+\sum_{r=1}^R\frac{|\delta_r| e^{-\pi \gcd^2(m_r,\ell)/m_r}}{\left(1-e^{-\pi \gcd^2(m_r,\ell)/m_r}\right)^2}\right)\frac{2}{kN}\ k\\
&= 2\Delta_4(\ell)\;N^{-2\Delta_1-1}e^{2\pi\varrho \left(n+\frac{\Delta_2}{24}\right)} \left(\sum_{\substack{1\le k\le N\\k\equiv_L \ell}}k^{\Delta_1}\right) \\
&\quad\quad\quad\quad\quad\quad\quad\quad\times  \exp\left(\frac{\pi\Delta_3(\ell)}{24}+\sum_{r=1}^R\frac{|\delta_r| e^{-\pi \gcd^2(m_r,\ell)/m_r}}{\left(1-e^{-\pi \gcd^2(m_r,\ell)/m_r}\right)^2}\right)\\
&\le 2\Delta_4(\ell)\;e^{2\pi\varrho \left(n+\frac{\Delta_2}{24}\right)}\;\Xi_{\Delta_1}(N)\exp\left(\frac{\pi\Delta_3(\ell)}{24}+\sum_{r=1}^R\frac{|\delta_r| e^{-\pi \gcd^2(m_r,\ell)/m_r}}{\left(1-e^{-\pi \gcd^2(m_r,\ell)/m_r}\right)^2}\right).
\end{align*}
Here we use the trivial bound
$$N^{-2\Delta_1-1} \sum_{\substack{1\le k\le N\\k\equiv_L \ell}}k^{\Delta_1}\le \Xi_{\Delta_1}(N),$$
where $\Xi_{\Delta_1}(N)$ is defined in \eqref{eq:xi-rep}.

To summarize, we have

\begin{lemma}\label{le:bound-S-l--}
For $\ell\in\cl_{\le 0}$, it holds that
\begin{equation}
|S_\ell|\le 2\Delta_4(\ell)\;e^{2\pi\varrho \left(n+\frac{\Delta_2}{24}\right)}\;\Xi_{\Delta_1}(N)\exp\left(\frac{\pi\Delta_3(\ell)}{24}+\sum_{r=1}^R\frac{|\delta_r| e^{-\pi \gcd^2(m_r,\ell)/m_r}}{\left(1-e^{-\pi \gcd^2(m_r,\ell)/m_r}\right)^2}\right).
\end{equation}
\end{lemma}

\subsection{Estimating $S_{\ell}$ with $\ell\in\cl_{> 0}$}

Assume that $\ell\in\cl_{> 0}$. Here $g(n)$ is dominated by these $S_{\ell}$.

We deduce from \eqref{eq:G-m-gamma} that
\begin{align}
S_\ell&=\sum_{\substack{1\le k\le N\\k\equiv_L \ell}}\sum_{\substack{0\le h< k\\ (h,k)=1}} \ee{-\frac{nh}{k}} \int_{\xi_{h,k}} G\big(\ee{\tau}\big)\ee{-n \phi} e^{2 \pi n \varrho}\ d\phi \nonumber\\
& = \sum_{\substack{1\le k\le N\\k\equiv_L \ell}}\sum_{\substack{0\le h< k\\ (h,k)=1}} \ee{-\frac{nh}{k}} \int_{\xi_{h,k}} \prod_{r=1}^R F\big(\ee{\gamma_{(m_r,k)}(m_r\tau)}\big)^{-\delta_{r}} \nonumber\\
&\quad\quad\quad\quad\quad\quad\quad\quad\quad\quad\quad \times \omega_{h,k}\; \Delta_4(\ell) \; e^{\frac{\pi}{12 k}\left(\frac{\Delta_3(\ell)}{z}+\Delta_2 z\right)}\; z^{\Delta_1}\;\ee{-n \phi} e^{2 \pi n \varrho} \ d\phi \nonumber\\
&= \sum_{\substack{1\le k\le N\\k\equiv_L \ell}}\sum_{\substack{0\le h< k\\ (h,k)=1}} \ee{-\frac{nh}{k}} \int_{\xi_{h,k}} \omega_{h,k}\; \Delta_4(\ell) \; e^{\frac{\pi}{12 k}\left(\frac{\Delta_3(\ell)}{z}+\Delta_2 z\right)}\; z^{\Delta_1}\;\ee{-n \phi} e^{2 \pi n \varrho}\ d\phi \nonumber\\
&\quad + \sum_{\substack{1\le k\le N\\k\equiv_L \ell}}\sum_{\substack{0\le h< k\\ (h,k)=1}} \ee{-\frac{nh}{k}} \int_{\xi_{h,k}} \left(\prod_{r=1}^R F\big(\ee{\gamma_{(m_r,k)}(m_r\tau)}\big)^{-\delta_{r}}-1\right) \nonumber\\
&\quad\quad\quad\quad\quad\quad\quad\quad\quad\quad\quad \times \omega_{h,k}\; \Delta_4(\ell) \; e^{\frac{\pi}{12 k}\left(\frac{\Delta_3(\ell)}{z}+\Delta_2 z\right)}\; z^{\Delta_1}\;\ee{-n \phi} e^{2 \pi n \varrho} \ d\phi \nonumber\\
&=:T_{\ell,1}+T_{\ell,2}.\label{eq:a-S2}
\end{align}

We first deal with $T_{\ell,2}$, which is relatively small comparing with $T_{\ell,1}$. The following bound is necessary.

\begin{lemma}\label{le:bound-1-1}
If $\Delta_3(\ell)>0$, we have, for $k\equiv \ell\pmod{L}$,
\begin{align}
&\left|\left(\prod_{r=1}^R F\big(\ee{\gamma_{(m_r,k)}(m_r\tau)}\big)^{-\delta_{r}}-1\right)e^{\frac{\pi}{12 k}\frac{\Delta_3(\ell)}{z}}\right|\nonumber\\
&\quad\le e^{\frac{\pi \Delta_3(\ell)}{24}}\left(\exp\left(\sum_{r=1}^R\frac{|\delta_r| e^{-\pi \gcd^2(m_r,\ell)/m_r}}{\left(1-e^{-\pi \gcd^2(m_r,\ell)/m_r}\right)^2}\right)-1\right).\label{eq:bound-1-1}
\end{align}
\end{lemma}

\begin{proof}
It follows from Lemma \ref{le:rough-bound} that
\begin{align*}
\left|\prod_{r=1}^R F\big(\ee{\gamma_{(m_r,k)}(m_r\tau)}\big)^{-\delta_{r}}-1\right|\le \exp\left(\sum_{r=1}^R\frac{|\delta_r| e^{-2\pi \Im(\gamma_{(m_r,k)}(m_r\tau))}}{\left(1-e^{-2\pi \Im(\gamma_{(m_r,k)}(m_r\tau))}\right)^2}\right)-1.
\end{align*}
We also know from \eqref{eq:gamma} that
$$\Im(\gamma_{(m_r,k)}(m_r\tau))=\frac{\gcd^2(m_r,\ell)}{m_rk}\Re\left(\frac{1}{z}\right).$$
Hence
\begin{align}
&\left|\left(\prod_{r=1}^R F\big(\ee{\gamma_{(m_r,k)}(m_r\tau)}\big)^{-\delta_{r}}-1\right)e^{\frac{\pi}{12 k}\frac{\Delta_3(\ell)}{z}}\right| \nonumber\\
&\quad= e^{\frac{\pi \Delta_3(\ell)}{12k}\Re\left(\frac{1}{z}\right)}\left|\prod_{r=1}^R F\big(\ee{\gamma_{(m_r,k)}(m_r\tau)}\big)^{-\delta_{r}}-1\right| \nonumber\\
&\quad\le e^{\frac{\pi \Delta_3(\ell)}{12k}\Re\left(\frac{1}{z}\right)}\left(\exp\left(\sum_{r=1}^R\frac{|\delta_r| e^{\frac{-2\pi \gcd^2(m_r,\ell)}{m_r k}\Re\left(\frac{1}{z}\right)}}{\left(1-e^{\frac{-2\pi \gcd^2(m_r,\ell)}{m_r k}\Re\left(\frac{1}{z}\right)}\right)^2}\right)-1\right). \label{eq:bound-1-1-1}
\end{align}

Note that $\Re\left(\frac{1}{z}\right)\ge \frac{k}{2}$. If we put
$$x=e^{-\frac{\pi \Delta_3(\ell)}{12k}\Re\left(\frac{1}{z}\right)},$$
then $x\in (0,e^{-\pi \Delta_3(\ell)/24}]\subseteq (0,1)$ since $\Delta_3(\ell)>0$. Let
$$u_r=\frac{24\gcd^2(m_r,\ell)}{\Delta_3(\ell) m_r}.$$
Recalling \eqref{eq:assump}
$$\frac{\gcd^2(m_r,\ell)}{m_r}\ge \frac{\Delta_3(\ell)}{24}>0,$$
we have $u_r\ge 1$ for all $r=1,\ldots,R$.

Now \eqref{eq:bound-1-1-1} becomes
\begin{align*}
\left|\left(\prod_{r=1}^R F\big(\ee{\gamma_{(m_r,k)}(m_r\tau)}\big)^{-\delta_{r}}-1\right)e^{\frac{\pi}{12 k}\frac{\Delta_3(\ell)}{z}}\right|\le \frac{1}{x}\left(\exp\left(\sum_{r=1}^R\frac{|\delta_r|x^{u_r}}{(1-x^{u_r})^2}\right)-1\right).
\end{align*}

Let
$$W(x)=\frac{1}{x}\left(\exp\left(\sum_{r=1}^R\frac{|\delta_r|x^{u_r}}{(1-x^{u_r})^2}\right)-1\right).$$
To prove \eqref{eq:bound-1-1}, it suffices to show that $W(x)$ is a non-decreasing function of $x$ on $(0,e^{-\pi \Delta_3(\ell)/24}]$ since the right-hand side of \eqref{eq:bound-1-1} is exactly $W(e^{-\pi \Delta_3(\ell)/24})$. It is equivalent to show that $W'(x)\ge 0$ on this interval. We have
\begin{align*}
W'(x)=&-\frac{1}{x^2}\left(\exp\left(\sum_{r=1}^R\frac{|\delta_r|x^{u_r}}{(1-x^{u_r})^2}\right)-1\right)\\
&+\frac{1}{x}\exp\left(\sum_{r=1}^R\frac{|\delta_r|x^{u_r}}{(1-x^{u_r})^2}\right)\sum_{r=1}^R\left(\frac{2|\delta_r|u_r x^{-1+2u_r}}{(1-x^{u_r})^3}+\frac{|\delta_r|u_r x^{-1+u_r}}{(1-x^{u_r})^2}\right).
\end{align*}
It suffices to show that
$$\exp\left(-\sum_{r=1}^R\frac{|\delta_r|x^{u_r}}{(1-x^{u_r})^2}\right)+\sum_{r=1}^R\left(\frac{2|\delta_r|u_r x^{2u_r}}{(1-x^{u_r})^3}+\frac{|\delta_r|u_r x^{u_r}}{(1-x^{u_r})^2}\right)\ge 1.$$

We next observe that
$$\sum_{r=1}^R\left(\frac{2|\delta_r|u_r x^{2u_r}}{(1-x^{u_r})^3}+\frac{|\delta_r|u_r x^{u_r}}{(1-x^{u_r})^2}\right)\ge \sum_{r=1}^R \frac{|\delta_r|x^{u_r}}{(1-x^{u_r})^2}.$$
This is valid as for $r=1,\ldots,R$, we have
\begin{align*}
\frac{2|\delta_r|u_r x^{2u_r}}{(1-x^{u_r})^3}+\frac{|\delta_r|u_r x^{u_r}}{(1-x^{u_r})^2}-\frac{|\delta_r|x^{u_r}}{(1-x^{u_r})^2} = \frac{|\delta_r|x^{u_r}(-1+u_r+x^{u_r}+u_r x^{u_r})}{(1-x^{u_r})^3}\ge 0
\end{align*}
since $u_r\ge 1$ and $x\in (0,e^{-\pi \Delta_3(\ell)/24}]$.

Letting
$$y=\sum_{r=1}^R \frac{|\delta_r|x^{u_r}}{(1-x^{u_r})^2},$$
we see that $y>0$ when $x\in (0,e^{-\pi \Delta_3(\ell)/24}]$.

Now it suffices to show that $e^{-y}+y\ge 1$, which is obvious as $e^{-y}+y$ is increasing for $y>0$. We therefore arrive at the desired result.
\end{proof}

\begin{remark}
It is helpful to mention more about the necessity of the assumption \eqref{eq:assump}. Suppose that there exist some $m_r$ and $\ell$ such that
$$\frac{\Delta_3(\ell)}{24}>\frac{\gcd^2(m_r,\ell)}{m_r}>0,$$
then in this case $0<u_r<1$. It is not hard to compute that
$$\lim_{x\to 0}W(x)=\infty.$$
This means that $W(x)$ is not bounded on the interval $(0,e^{-\pi \Delta_3(\ell)/24}]$, which is not what we hope.
\end{remark}

It follows from Lemma \ref{le:bound-1-1} that
\begin{align*}
|T_{\ell,2}|&\le  \sum_{\substack{1\le k\le N\\k\equiv_L \ell}}\sum_{\substack{0\le h< k\\ (h,k)=1}} \int_{\xi_{h,k}} \Delta_4(\ell)\; k^{\Delta_1} N^{-2\Delta_1} e^{2\pi\varrho \left(n+\frac{\Delta_2}{24}\right)}e^{\frac{\pi \Delta_3(\ell)}{24}}\\
&\quad\quad\quad\quad\quad\quad\quad \times \left(\exp\left(\sum_{r=1}^R\frac{|\delta_r| e^{-\pi \gcd^2(m_r,\ell)/m_r}}{\left(1-e^{-\pi \gcd^2(m_r,\ell)/m_r}\right)^2}\right)-1\right)\ d\phi\\
&\le  \sum_{\substack{1\le k\le N\\k\equiv_L \ell}}\sum_{\substack{0\le h< k\\ (h,k)=1}}  \Delta_4(\ell)\; k^{\Delta_1} N^{-2\Delta_1} e^{2\pi\varrho \left(n+\frac{\Delta_2}{24}\right)}e^{\frac{\pi \Delta_3(\ell)}{24}}\\
&\quad\quad\quad\quad\quad\quad\quad \times \left(\exp\left(\sum_{r=1}^R\frac{|\delta_r| e^{-\pi \gcd^2(m_r,\ell)/m_r}}{\left(1-e^{-\pi \gcd^2(m_r,\ell)/m_r}\right)^2}\right)-1\right)\frac{2}{kN}\\
&\le  \sum_{\substack{1\le k\le N\\k\equiv_L \ell}} \Delta_4(\ell)\; k^{\Delta_1} N^{-2\Delta_1} e^{2\pi\varrho \left(n+\frac{\Delta_2}{24}\right)}e^{\frac{\pi \Delta_3(\ell)}{24}}\\
&\quad\quad\quad\quad\quad\quad\quad \times \left(\exp\left(\sum_{r=1}^R\frac{|\delta_r| e^{-\pi \gcd^2(m_r,\ell)/m_r}}{\left(1-e^{-\pi \gcd^2(m_r,\ell)/m_r}\right)^2}\right)-1\right)\frac{2}{kN}\; k\\
&= 2\Delta_4(\ell)\;N^{-2\Delta_1-1}e^{2\pi\varrho \left(n+\frac{\Delta_2}{24}\right)} e^{\frac{\pi \Delta_3(\ell)}{24}} \left(\sum_{\substack{1\le k\le N\\k\equiv_L \ell}}k^{\Delta_1}\right) \\
&\quad\quad\quad\quad\quad\quad\quad \times \left(\exp\left(\sum_{r=1}^R\frac{|\delta_r| e^{-\pi \gcd^2(m_r,\ell)/m_r}}{\left(1-e^{-\pi \gcd^2(m_r,\ell)/m_r}\right)^2}\right)-1\right)\\
&\le 2\Delta_4(\ell)\;e^{\frac{\pi \Delta_3(\ell)}{24}}e^{2\pi\varrho \left(n+\frac{\Delta_2}{24}\right)}\;\Xi_{\Delta_1}(N)\\
&\quad\quad\quad\quad\quad\quad\quad \times \left(\exp\left(\sum_{r=1}^R\frac{|\delta_r| e^{-\pi \gcd^2(m_r,\ell)/m_r}}{\left(1-e^{-\pi \gcd^2(m_r,\ell)/m_r}\right)^2}\right)-1\right).
\end{align*}

To summarize, we have

\begin{lemma}\label{le:T-l-2}
For $\ell\in\cl_{> 0}$, it holds that
\begin{align}
|T_{\ell,2}|&\le 2\Delta_4(\ell)\;e^{\frac{\pi \Delta_3(\ell)}{24}}e^{2\pi\varrho \left(n+\frac{\Delta_2}{24}\right)}\;\Xi_{\Delta_1}(N) \nonumber\\
&\quad \times \left(\exp\left(\sum_{r=1}^R\frac{|\delta_r| e^{-\pi \gcd^2(m_r,\ell)/m_r}}{\left(1-e^{-\pi \gcd^2(m_r,\ell)/m_r}\right)^2}\right)-1\right).
\end{align}
\end{lemma}

At last, we estimate $T_{\ell,1}$.

\begin{lemma}\label{le:T-l-1}
For $\ell\in\cl_{> 0}$, it holds that
\begin{align}
T_{\ell,1}&=D_\ell+2\pi\;\Delta_4(\ell) \left(\frac{24n+\Delta_2}{\Delta_3(\ell)}\right)^{-\frac{\Delta_1+1}{2}} \nonumber\\
&\quad\quad\quad\quad \times\sum_{\substack{1\le k\le N\\k\equiv_L \ell}}\frac{I_{-\Delta_1-1}\left(\frac{\pi }{6k}\sqrt{\Delta_3(\ell)(24n+\Delta_2)}\right)}{k} \sum_{\substack{0\le h< k\\ (h,k)=1}} \omega_{h,k}\; \ee{-\frac{nh}{k}},
\end{align}
where
\begin{align}
|D_\ell|\le \frac{2^{-\Delta_1}\pi^{-1}e^{\frac{\Delta_3(\ell)\pi}{3}}N^{-\Delta_1+2}}{n+\frac{\Delta_2}{24}}e^{2\pi\varrho \left(n+\frac{\Delta_2}{24}\right)}.
\end{align}
\end{lemma}

\begin{proof}
It follows from Lemma \ref{le:key-int} that
\begin{align*}
T_{\ell,1}&= \sum_{\substack{1\le k\le N\\k\equiv_L \ell}}\sum_{\substack{0\le h< k\\ (h,k)=1}} \ee{-\frac{nh}{k}} \int_{\xi_{h,k}} \omega_{h,k}\; \Delta_4(\ell) \; e^{\frac{\pi}{12 k}\left(\frac{\Delta_3(\ell)}{z}+\Delta_2 z\right)}z^{\Delta_1}\ee{-n \phi} e^{2 \pi n \varrho}\ d\phi\\
&=\sum_{\substack{1\le k\le N\\k\equiv_L \ell}}\sum_{\substack{0\le h< k\\ (h,k)=1}} \ee{-\frac{nh}{k}} \omega_{h,k}\; \Delta_4(\ell) \int_{\xi_{h,k}} e^{\frac{\pi}{12 k}\left(\frac{\Delta_3(\ell)}{z}+\Delta_2 z\right)}z^{\Delta_1}\ee{-n \phi} e^{2 \pi n \varrho}\ d\phi\\
&=D_\ell+2\pi\;\Delta_4(\ell) \left(\frac{24n+\Delta_2}{\Delta_3(\ell)}\right)^{-\frac{\Delta_1+1}{2}} \\
&\quad\quad\quad\quad \times\sum_{\substack{1\le k\le N\\k\equiv_L \ell}}\frac{I_{-\Delta_1-1}\left(\frac{\pi }{6k}\sqrt{\Delta_3(\ell)(24n+\Delta_2)}\right)}{k} \sum_{\substack{0\le h< k\\ (h,k)=1}} \omega_{h,k}\; \ee{-\frac{nh}{k}}.
\end{align*}
Here
\begin{align*}
|D_\ell|&\le \sum_{\substack{1\le k\le N\\k\equiv_L \ell}}\sum_{\substack{0\le h< k\\ (h,k)=1}} \frac{2^{-\Delta_1}\pi^{-1}e^{\frac{\Delta_3(\ell)\pi}{3}}N^{-\Delta_1}}{n+\frac{\Delta_2}{24}}e^{2\pi\varrho \left(n+\frac{\Delta_2}{24}\right)}\\
&\le \frac{2^{-\Delta_1}\pi^{-1}e^{\frac{\Delta_3(\ell)\pi}{3}}N^{-\Delta_1+2}}{n+\frac{\Delta_2}{24}}e^{2\pi\varrho \left(n+\frac{\Delta_2}{24}\right)}.
\end{align*}
\end{proof}

\subsection{The asymptotic formula of $g(n)$}

We know from Lemma \ref{le:T-l-1} that the main term of $g(n)$ is
\begin{align}
&\sum_{\ell\in \cl_{>0}} 2\pi\;\Delta_4(\ell) \left(\frac{24n+\Delta_2}{\Delta_3(\ell)}\right)^{-\frac{\Delta_1+1}{2}} \nonumber\\
&\quad\quad\quad \times\sum_{\substack{1\le k\le N\\k\equiv_L \ell}}\frac{I_{-\Delta_1-1}\left(\frac{\pi }{6k}\sqrt{\Delta_3(\ell)(24n+\Delta_2)}\right)}{k} \sum_{\substack{0\le h< k\\ (h,k)=1}} \omega_{h,k}\; \ee{-\frac{nh}{k}}.
\end{align}

Furthermore, the total error term is
\begin{align}
|E(n)|&\le \sum_{\ell\in \cl_{\le 0}}|S_\ell| + \sum_{\ell\in \cl_{>0}}\Big(|T_{\ell,2}|+|D_\ell|\Big) \nonumber\\
&\le \frac{2^{-\Delta_1}\pi^{-1}N^{-\Delta_1+2}}{n+\frac{\Delta_2}{24}}e^{2\pi\varrho \left(n+\frac{\Delta_2}{24}\right)} \sum_{\ell\in \cl_{>0}} e^{\frac{\Delta_3(\ell)\pi}{3}} \nonumber\\
&\quad+ 2e^{2\pi\varrho \left(n+\frac{\Delta_2}{24}\right)}\;\Xi_{\Delta_1}(N) \nonumber\\
&\quad\quad\times\Bigg(\sum_{1\le \ell\le L} \Delta_4(\ell) \exp\left(\frac{\pi\Delta_3(\ell)}{24}+\sum_{r=1}^R\frac{|\delta_r| e^{-\pi \gcd^2(m_r,\ell)/m_r}}{\left(1-e^{-\pi \gcd^2(m_r,\ell)/m_r}\right)^2}\right) \nonumber\\
&\quad\quad\quad\quad- \sum_{\ell\in \cl_{>0}} \Delta_4(\ell) e^{\frac{\pi \Delta_3(\ell)}{24}} \Bigg),\label{eq:exp-E-bound}
\end{align}
where we use Lemmas \ref{le:bound-S-l--}, \ref{le:T-l-2} and \ref{le:T-l-1}.

At last, we set
$$N=\left\lfloor \sqrt{2\pi \left(n+\frac{\Delta_2}{24}\right)}\right\rfloor.$$
It is easy to check that
\begin{align*}
|E(n)|\ll_{\bfm,\bfd}\begin{cases}
1 & \text{if $\Delta_1=0$},\\
\left(n+\frac{\Delta_2}{24}\right)^{1/4} & \text{if $\Delta_1=-\frac{1}{2}$},\\
\left(n+\frac{\Delta_2}{24}\right)^{1/2} \log \left(n+\frac{\Delta_2}{24}\right)& \text{if $\Delta_1=-1$},\\
\left(n+\frac{\Delta_2}{24}\right)^{-\Delta_1-1/2} & \text{if $\Delta_1\le -\frac{3}{2}$}.\\
\end{cases}
\end{align*}

We therefore arrive at Theorem \ref{th:main}.

\section{An application}

We end this paper with an application of the main result. Here we will study the asymptotic behavior of the following eta-quotient:
\begin{align}
G_1(q)&=\frac{\f{2}^3}{\f{1}^2\f{10}}=\sum_{n\ge 0} g_1(n)q^n.\label{eq:gf-G1}
\end{align}
This eta-quotient is closely related to the rank statistics for cubic partition pairs. We refer the readers to \cite{Kim2011} for details. In particular, \eqref{eq:gf-G1} appears on p.~5 of \cite{Kim2011} (the third line of (3.2) therein).

\begin{theorem}
We have, as $n\to\infty$,
\begin{equation}
g_1(n)\sim 3^{\frac{3}{4}} 5^{\frac{1}{4}} (24n-6)^{-\frac{3}{4}} \exp\left(\frac{\pi}{2\sqrt{15}}\sqrt{24n-6}\right).
\end{equation}
\end{theorem}

\begin{proof}
We have $\bfm=(1,2,10)$ and $\bfd=(-2,3,-1)$. It is straightforward to compute that $\Delta_1=0$ and $\Delta_2=-6$. We also have $L=10$. The values of $\Delta_3(\ell)$ and $\Delta_4(\ell)$ for $1\le \ell\le L$ are listed in Table \ref{ta:g2-1}. Hence $\cl_{>0}=\{1,3,5,7,9,10\}$.
\begin{table}[ht]\caption{The values of $\Delta_3(\ell)$ and $\Delta_4(\ell)$ for $1\le \ell\le L$}\label{ta:g2-1}
\centering
\def\arraystretch{1.5}
\begin{tabular}{ccccccccccc}
\hline
$\ell$ & 1 & 2& 3& 4& 5& 6& 7& 8& 9& 10\\
$\Delta_3(\ell)$ & $\frac{3}{5}$& $-\frac{18}{5}$& $\frac{3}{5}$& $-\frac{18}{5}$& $3$& $-\frac{18}{5}$& $\frac{3}{5}$& $-\frac{18}{5}$& $\frac{3}{5}$& $6$\\
$\Delta_4(\ell)$ & $\frac{\sqrt{5}}{2}$& $\sqrt{5}$& $\frac{\sqrt{5}}{2}$& $\sqrt{5}$& $\frac{1}{2}$& $\sqrt{5}$& $\frac{\sqrt{5}}{2}$& $\sqrt{5}$& $\frac{\sqrt{5}}{2}$& $1$\\
\hline
\end{tabular}
\end{table}

We know from Theorem \ref{th:main} along with \eqref{Bessel-order} and \eqref{eq:tail-est} that $g_1(n)$ is dominated by the largest
$$I_{-\Delta_1-1}\left(\frac{\pi }{6k}\sqrt{\Delta_3(\ell)(24n+\Delta_2)}\right)$$
provided that this term does not vanish. From Table \ref{ta:g2-2}, we see that when $k=1$, the previous modified Bessel function of the first kind has the largest order.
\begin{table}[ht]\caption{The values of $\frac{\pi }{6k}\sqrt{\Delta_3(\ell)}$}\label{ta:g2-2}
\centering
\def\arraystretch{1.5}
\begin{tabular}{ccccccc}
\hline
$k$ & 1 & 3& 5& 7& 9& 10\\
$\frac{\pi }{6k}\sqrt{\Delta_3(\ell)}$ & $\frac{\pi}{2\sqrt{15}}$& $\frac{\pi}{6\sqrt{15}}$& $\frac{\pi}{10\sqrt{3}}$& $\frac{\pi}{14\sqrt{15}}$& $\frac{\pi}{18\sqrt{15}}$& $\frac{\pi}{10\sqrt{6}}$\\
\hline
\end{tabular}
\end{table}

For $k=1$ (and hence $\ell=1$), we further compute that
$$\frac{1}{k}\sum_{\substack{0\le h< k\\ (h,k)=1}} \omega_{h,k}\; \ee{-\frac{nh}{k}}=\cos(2\pi n)=1$$
for all $n\in\mathbb{Z}_{>0}$. We also have
$$2\pi\;\Delta_4(\ell) \left(\frac{24n+\Delta_2}{\Delta_3(\ell)}\right)^{-\frac{\Delta_1+1}{2}}=3^{\frac{1}{2}}\pi (24n-6)^{-\frac{1}{2}}.$$

We therefore deduce that, as $n\to\infty$,
\begin{align*}
g_1(n)&\sim 3^{\frac{1}{2}}\pi (24n-6)^{-\frac{1}{2}} I_{-1}\left(\frac{\pi}{2\sqrt{15}}\sqrt{24n-6}\right)\\
&\sim 3^{\frac{3}{4}} 5^{\frac{1}{4}} (24n-6)^{-\frac{3}{4}} \exp\left(\frac{\pi}{2\sqrt{15}}\sqrt{24n-6}\right),
\end{align*}
where we use \eqref{Bessel-order}.
\end{proof}


\bibliographystyle{amsplain}

\end{document}